\documentclass[12pt]{article}

\usepackage[cp1251]{inputenc}
\usepackage{color}
\usepackage{xcolor}

\usepackage{amsmath,amsthm,amsfonts,amssymb}
\usepackage{array}
\theoremstyle{plain}
\newtheorem{theorem}{Theorem}
\theoremstyle{plain}
\newtheorem{prop}{Proposition}
\theoremstyle{plain}

\theoremstyle{plain}
\newtheorem{cor}{Corollary}
\theoremstyle{definition}
\newtheorem{defin}{Definition}
\theoremstyle{remark}

\theoremstyle{remark}

\theoremstyle{remark}

\begin{document}

UDK 517.982.27+519.2

\begin{center}
{\bf On unconditionality of fractional Rademacher chaos in symmetric spaces.\\[5mm]}
\end{center}

\begin{center}
{\bf S.V. Astashkin
\footnote[1]{The work of the first named author was completed as a part of the implementation of the development program of the Volga Region Scientific and Educational Mathematical Center (agreement no. 075-02-2023-931).}, K.V. Lykov\\[5mm]}
\end{center}

{\bf Key words}: Rademacher functions, Rademacher chaos, symmmetric space, combinatorial dimension, unconditional convergence.

{\bf Abstract}: We study density estimates of an index set $\mathcal{A}$, under which unconditionality (or even a weaker property of the random unconditional divergence) of the corresponding Rademacher fractional chaos $\{r_{j_1}(t)\cdot r_{j_2}(t)\cdot\dots\cdot r_{j_d}(t)\}_{(j_1,j_2,\dots,j_d)\in \mathcal{A}}$ in a symmetric space $X$ implies its equivalence in $X$ to the canonical basis in $\ell_2$. In the special case of Orlicz spaces $L_M$,  unconditionality of this system is also equivalent to the fact that a certain exponential Orlicz space embeds into $L_M$.




\section{Introduction}
\label{Intro}

As usual, the Rademacher functions are defined as follows: if $0\leqslant t\leqslant 1,$ then 
$$r_j(t):=(-1)^{[2^j t]},\quad j=1,2,\dots,
$$
where $[x]$ denotes the integer part of a real number $x$ (i.e., the most integer, which does not exceed $x$). According to the classical Khintchine inequality \cite{khint} (cf. also \cite{AsBook}), for any $p\geqslant 1$ there exists a constant $C_p$ such that for arbitrary $a_j\in\mathbb{R}$, $j=1,2,\dots$, we have
\begin{equation}
\label{basic}
{\left\|\sum_{j=1}^\infty
a_jr_j\right\|}_{L_p[0,1]}\leqslant C_p\left(\sum_{j=1}^\infty a_j^2\right)^{1/2}.
\end{equation}
It is well known that $C_p\le \sqrt{p}$ (the optimal constants for this inequality have been found by Haagerup in \cite{Haag}). In the opposite direction, in \cite{Szarek}, Szarek proved that for all $p\geqslant 1$ and $a_j\in\mathbb{R}$, $j=1,2,\dots$, it holds
\begin{equation}
\label{basic1}
\frac{1}{\sqrt{2}}\left(\sum_{j=1}^\infty a_j^2\right)^{1/2}\leqslant{\left\|\sum_{j=1}^\infty
a_jr_j\right\|}_{L_p[0,1]}.
\end{equation}

These inequalities have caused an enormous number of investigations and generalizations, have found numerous applications in various fields of  analysis. Recall that Khintchine proved inequality \eqref{basic}, ''by pursuing the goal of finding the 'right' rate of convergence in the strong law of large numbers of Borel'' \cite{PeshShir}. At the same time, from the point of view of the geometry of Banach spaces  inequalities \eqref{basic} and \eqref{basic1} indicate that the spaces $L_p[0,1]$, $1\le p<\infty$, while not being Hilbert for $ p\neq 2$, contain subspaces isomorphic to $\ell_2$. A characterization of  symmetric spaces $X$, in which the sequence $\{r_j\}_{j=1}^\infty$ is equivalent to the canonical basis in $\ell_2$, was obtained by Rodin and Semenov in the paper \cite{RS}. They proved that this equivalence holds if and only if $X$ contains the separable part of the Orlicz space $\mathrm{Exp}L^{2}$, generated by the function $N_2(u)=e^{u^{2}}-1$. In \cite{As1998}, a similar question was studied for the system $\{r_{j_1}(t)\cdot r_{j_2}(t)\}_{j_1>j_2}$ of products of Rademacher functions, which is called usually the second-order Rademacher chaos. Specifically, it was shown that this system is equivalent in $X$ to the canonical basis in $\ell_2$ if and only if $X$ contains the separable part of the Orlicz space $\mathrm{Exp}L$, generated by the function $N_1(u)=e^u-1$. Moreover, both these properties turned out to be equivalent to the formally weaker (than the equivalence to the canonical basis in $\ell_2$) property of unconditionality of the basic sequence $\{r_{j_1}(t)\cdot r_{j_2}(t)\}_{j_1>j_2}$ in $X$ \cite{As2000}. Note that the Rademacher system itself is an unconditional (and even symmetric with constant 1) basic sequence in any symmetric space (see e.g. \cite[Proposition~2.2]{AsBook}). The next step in the study of the behaviour of the Rademacher chaos in symmetric spaces was made by the authors of this paper by using the important concept of combinatorial dimension developed earlier by Blei (see \cite{Blei79,Blei84,Blei-Korner,Blei,BleiJanson}) . Namely, in \cite{AL_AA}, it was shown that the above results in \cite{As1998} and \cite{As2000} can be extended to a non-complete chaos $\{r_{j_1}(t)\cdot r_{j_2}(t)\cdot\dots\cdot r_{j_d}(t)\}_{(j_1,j_2,\dots,j_d)\in \mathcal{A}}$ whenever combinatorial dimension  of the corresponding index set $\mathcal{A}\subset \mathbb{N}^d$ is equal to $d$.

The main aim of this paper is to determine conditions on an index set $\mathcal{A}$, under which unconditionality of the system $\{r_{j_1}(t)\cdot r_{j_2}(t)\cdot\dots\cdot r_{j_d}(t)\}_{(j_1,j_2,\dots,j_d)\in \mathcal{A}}$ in a symmetric space $X$ guaranties its equivalence in $X$ to the canonical basis in $\ell_2$. In particular, bearing in mind the above-mentioned specifics in the behaviour of the chaos in comparison with the Rademacher system itself, we investigate quantitative connections between properties of a subsystem of the Rademacher chaos and combinatorial dimension of the corresponding index set. To achieve this goal, we slightly modify the notion of combinatorial dimension, using one-sided density estimates for an index set $\mathcal{A}$, which allows us to significantly expand the scope of estimates of the form \eqref{basic}. 

It is worth to noting the following new effect that appeared in this paper. As Theorem~\ref{RUDgen} shows, certain density estimates of an index set guarantee that "remoteness"\;of a symmetric space $X$ from the "extreme"\;space $L_\infty$ is a consequence even of the so-called random unconditional divergence (RUD) property of the system $\{r_{j_1}(t)\cdot r_{j_2}(t)\cdot\dots\cdot r_{j_d}(t)\}_{(j_1,j_2 ,\dots,j_d)\in \mathcal{A}}$ in $X$, which is essentially weaker than unconditionality. Thus, in this case such a system possesses the RUD property  in a symmetric space $X$ if and only if it is equivalent in $X$ to the canonical basis in $\ell_2$ (see Theorem \ref{FractChaos}). 

In the special case of Orlicz spaces $L_M$ basic properties of a system $\{r_{j_1}(t)\cdot r_{j_2}(t)\cdot\dots\cdot r_{j_d}(t)\}_ {(j_1,j_2,\dots,j_d)\in \mathcal{A}}$ can be characterized also, as well as in \cite{RS,As1998,AL_AA}, in terms of continuous embeddings of certain exponential Orlicz spaces into $L_M$ (Theorem~\ref{FractChaosOrl}). Note that somewhat close results to Theorem~\ref{FractChaosOrl} were obtained earlier by Blei and Ge \cite{Blei-Ge1,Blei-Ge2}. Instead of issues associated with unconditionality of a system, they make use of a more detailed analysis of combinatorial dimension of the corresponding index set.

In the concluding part of the paper, we show that every uniformly bounded Bessel system (in particular, any Rademacher chaos) in a symmetric space $X$, satisfying the embedding $\mathrm{Exp}L^{2}\subset X$, has the random unconditional convergence property, which is in a certain sense opposite to the RUD property. In addition, we give here a concrete example illustrating the interesting fact of "divergence"\:of the moment estimates of a Rademacher fractional chaos and its asymptotic behaviour (see also \cite{BleiJanson}).

\section{Preliminaries}

In what follows, any embedding of one Banach space into another is understood as continuous, i.e., $X_1\subset X_0$ means that from $x\in X_1$ it follows $x\in X_0$ and ${\|x\|}_{X_0}\leqslant C{\|x\|}_{X_1 }$ for some $C>0$. If the specific embedding constant $C$ is important,  we will additionally write $X_1\overset{C}{\subset} X_0$. An expression of the form $F_1\asymp F_2$ means that $cF_1\leqslant F_2\leqslant CF_1$ for some constants $c>0$ and $C>0$, which do not depend on all or a part of the arguments of $F_1$ and $ F_2$, and it should be clear from the context which arguments are involved.

By $|\cdot|$ we denote further the modulus of a number or a function, as well as cardinality of a set, depending on the context.

\subsection{Symmetric spaces}
\label{Symm}

A detailed exposition of the theory of symmetric spaces can be found in the monographs \cite{KPS,LT,BS}.

Let $\mathcal{S}$ be the set of (equivalence classes of) measurable almost everywhere finite real-valued functions on $[0,1]$ with the usual Lebesgue measure $\mu$. The {\it distribution function} of a function $x=x(t)\in\mathcal{S}$ is defined as follows: 
$$
n_x(\tau):=\mu \{t:\;x(t)>\tau\},\;\tau\in\mathbb{R}.
$$ 
We will say that two functions $x$ and $y$ are {\it equidistributed} if their distribution functions coincide, while these functions will be called {\it equimeasurable} if the functions $|x|$ and $|y|$ are equidistributed.

For any function $x=x(t)\in\mathcal{S}$ there exists a unique decreasing, left-continuous, non-negative function $x^*=x^*(t)$ on $[0,1]$, equimeasurable with $x$, which is referred as the {\it rearrangement} of $x$. It can be explicitly specified by the following formula \cite[p.~59]{KPS}: 
$$
x^*(t)=\inf\{\tau:\;n_{|x|}(\tau)<t\}.
$$


\begin{defin} A Banach space $X$, $X\subset \mathcal{S}$, is said to be {\it ideal} if the conditions $x\in X$, $y\in \mathcal S$ and $|y|\leqslant|x|$ imply $y\in X$ and ${\|y\|}_X\leqslant{\|x\|}_X$. A Banach ideal space $X$ is said to be {\it symmetric} if from the conditions $x\in X$, $y\in \mathcal S$ and $y^*=x^*$ it follows that $y\in X$ and ${\|y\|}_X={\|x\|}_X$.
\end{defin}

It follows from the definition that a symmetric space, along with each function $x$, also contains all the functions that are equimeasurable with $x$. 

Let us give now some examples of symmetric spaces on $[0,1]$.

As usual, the space $L_p=L_p[0,1]$, $1\leqslant p<\infty$, consists of all functions $x\in \mathcal S$, for which 
$$
{\|x\|}_p:=\left(\int\limits_0^1|x(t)|^p\,dt\right)^{{1}/{p}}<\infty.
$$
For $p>q$, we have $L_p\overset{1}{\subset} L_q$. In the limiting case as $p\to\infty$ we obtain the space $L_\infty$ with the norm
$$
{\|x\|}_\infty:=\mathrm{ess}\sup_{t\in[0,1]}|x(t)|=\inf\left\{C:\,\mu\{t\in [0,1]:\,|x(t)|>C\}=0\right\}.
$$

A natural generalization of the $L_p$-spaces are the Orlicz spaces. Let $M=M(u)$ be an {\it Orlicz function}, i.e., convex, non-negative function on $[0,\infty)$, not identically equal to zero, $M(0)=0$. 
The {\it Orlicz space} $L_M$ consists of all functions $x\in \mathcal S$, for which there exists $\lambda>0$ such that
$$
\int\limits_0^1 M\left(\frac{|x(t)|}{\lambda }\right)\,dt<\infty.
$$
The norm in $L_M$ is defined by
$$
{\|x \|} _{L_M}:=\inf\left\{\lambda>0 :\;\int\limits_0^1 M\left(\frac{|x(t)|}{\lambda }\right)\,dt\leqslant 1\right\}.
$$
In particular, $L_{M_p}=L_p$ isometrically if $M_p(u)=u^p$. By $\mathrm{Exp}L^{r}$, $r>0$, we will denote the exponential Orlicz space, generated by an Orlicz function $N_r(u)$ such that, for some $u_0>0$, we have $\log N_r(u)\asymp u^{r}$ if $u>u_0$.

Further, we will use repeatedly the following extrapolation description of exponential Orlicz spaces $\mathrm{Exp}L^{r}$ (see \cite[formulas (2)--(4)]{MO}, \cite[p.~83]{JM92} or \cite[Lemma X.18]{Blei}):
\begin{equation}
\label{description}
{\|x\|}_{{\rm Exp} L^{r}}\asymp \sup_{p\geqslant 1}\frac{\|x\|_p}{p^{1/r}}.
\end{equation}
For more details related to the Orlicz spaces see, for instance, the book \cite{KrasR}.

Let $\varphi$ be a continuous, increasing, concave function on $[0,1]$, $\varphi(0)=0$. The {\it Lorentz space} $\Lambda(\varphi)$ (resp. {\it Marcinkiewicz space} $\mathcal{M}(\varphi)$) consists of all functions $x\in\mathcal{S}$ such that 
$$
{\|x\|}_{\Lambda(\varphi)}:=\int\limits_0^1x^*(t)\,d\varphi(t)<\infty
$$
(resp. 
$$
{\|x\|}_{\mathcal{M}(\varphi)}:=\sup_{t\in(0,1]}\frac{\varphi(t)}{t}\int\limits_0^tx^*(s)\,ds<\infty).
$$

For any symmetric space $X$ on $[0,1]$ we have $L_\infty\subset X\subset L_1$ \cite[Theorem~II.4.1]{KPS}. The closure of $L_\infty$ in a symmetric space $X$ is referred as the {\it separable part} of $X$ and denoted by $X^\circ$. Note that $X^\circ$ is a symmetric space that is separable whenever $X\ne L_\infty$.

An important characteristic of a symmetric space $X$ is its {\it fundamental function} $\phi_X$ defined by
$$
\phi_X(t):={\|\chi_{(0,t)}\|}_X,\quad t\in[0,1].
$$
Throughout the paper, $\chi_A$ is the characteristic function (indicator) of a set $A\subset[0,1]$.
The fundamental function of a symmetric space is quasiconcave (i.e., $\phi_X(t)$ is increasing, $\phi_X(t)/t$ is decreasing and $\phi_X(0)=0$). Recall also that any quasiconcave function is equivalent to its smallest concave majorant (in the sense of the relation $\asymp$ defined above) \cite[see Corollary after Theorem~II.1.1]{KPS}. In particular,
$$
\phi_{\mathcal{M}(\varphi)}(t)=\phi_{\Lambda(\varphi)}(t)=\varphi(t),\quad\phi _{L_M}(t)= \frac{1}{M^{-1}\left(\frac{1}{t}\right)}.
$$

Note that, under certain conditions, the Orlicz and Marcinkiewicz spaces can coincide. Namely, $L_M=\mathcal{M}(\varphi)$ if and only if 
\begin{equation}\label{RavFundFunk}
\varphi(t)\asymp\frac{1}{M^{-1}(1/t)}
\end{equation}
and
\begin{equation}\label{RavProstranstv}
\int\limits_0^1M\left(\frac{\varepsilon}{\varphi(t)}\right)\,dt<\infty\quad\mbox{for some }\varepsilon>0
\end{equation}
(see \cite{Lorentz,Rutickiy}).

The Lorentz space $\Lambda(\varphi)$ has the following extremal property in the class of symmetric spaces: if $\phi_X(t)\leqslant C\varphi(t)$ for some $C>0$ and all $t\in[0,1]$, then $\Lambda(\varphi)\subset X$ (see \cite[Theorem~II.5.5]{KPS}). In particular, the Lorentz space $\Lambda(\varphi)$ is the narrowest among all symmetric spaces with the fundamental function $\varphi(t)$. The widest in the same class is the Marcinkiewicz space $\mathcal{M}(\varphi)$ \cite[Theorem~II.5.7]{KPS}. Thus, if a symmetric space $X$ is such that $\phi_X=\varphi$, then the following continuous embeddings are valid:
\begin{equation}\label{VlLorXMarc}
\Lambda(\varphi)\subset X\subset\mathcal{M}(\varphi).
\end{equation}


\subsection{Combinatorial dimension and $(\alpha,\beta)$-sets}

Based on the previously introduced concept of the fractional Cartesian product \cite{Blei79}, Blei came to the following definition of combinatorial dimension of a set (see \cite{Blei84}, as well as the monograph \cite[Chapter~XIII]{Blei}, which contains many interesting  applications of this notion). 

Let $d\in \mathbb{N}$ and $\mathbb{N}^d:=\mathbb{N}\times\mathbb{N}\times\ldots\times\mathbb{N}$ ($d$ factors), where $\mathbb{N}$ is the set of positive integers.

\begin{defin}\label{Dimension2}
A set $\mathcal{A}\subset\mathbb{N}^d$ is said to have {\it  combinatorial dimension $\alpha$} if \\
1) for an arbitrary $\beta>\alpha$ there exists $C_\beta>0$ such that for any $n\in\mathbb{N}$ and every collection of sets 
$B_1,B_2,\ldots,B_d\subset\mathbb{N}$, $|B_1|=|B_2|=\ldots=|B_d|=n$, we have
$$
|\mathcal{A}\cap (B_1\times B_2\times\ldots\times B_d)|<C_\beta n^\beta;
$$
2) for an arbitrary $\gamma<\alpha$ and $k\in\mathbb{N}$ there are $n>k$ and sets $B_1,B_2,\ldots,B_d\subset\mathbb{N}$, $|B_1|=|B_2|=\ldots=|B_d|=n$, such that
$$
|\mathcal{A}\cap (B_1\times B_2\times\ldots\times B_d)|> n^\gamma.
$$
\end{defin}
It is known that for each real $\alpha\in[1,d]$ there exists a set of combinatorial dimension $\alpha$ (see \cite{Blei-Korner} or \cite[Chapter~XIII]{Blei}).

Observe that Definition \ref{Dimension2} contains some asymmetry between the lower and upper density estimates for the set $\mathcal{A}$. We will use here the following modification of this definition, in which these estimates are considered separately.

\begin{defin}\label{a-set} Let $\mathcal{A}\subset\mathbb{N}^d$, $\alpha\geqslant 1$. We will say that a set $\mathcal{A}$ is a {\it super-$\alpha$-set} whenever the following condition holds: for some  $c_{\mathcal{A}}>0$ and each $n\in\mathbb{N}$ there are sets $B_1,B_2,\ldots, B_d$ such that $|B_j|=n$, $j=1,2,\ldots,d$, and
$$
|\mathcal{A}\cap (B_1\times B_2\times\ldots \times B_d)|\geqslant c_{\mathcal{A}}n^\alpha.
$$
\end{defin}

Let us emphasize that, in contrast to the part 2) of Definition \ref{Dimension2}, in Definition \ref{a-set} sets $B_1,B_2,\ldots, B_d$, satisfying the lower density estimate, exist for {\it each} positive integer $n$.


\begin{defin}\label{b-set} Let $\mathcal{A}\subset\mathbb{N}^d$, $\beta\leqslant d$. We will say that $\mathcal{A}$ is a {\it sub-$\beta$-set} if the following condition holds: for some $C_{\mathcal{A}}>0$, each $n\in\mathbb{N}$ and all sets $B_1,B_2,\ldots, B_d$, $|B_j|=n$, $j=1,2,\ldots,d$, we have
$$
|\mathcal{A}\cap (B_1\times B_2\times\ldots \times B_d)|\leqslant C_{\mathcal{A}}n^\beta.
$$
\end{defin}

\begin{defin}\label{ab-set} A set $\mathcal{A}\subset\mathbb{N}^d$,  which is both a super-$\alpha$-set and a sub-$\beta$-set, will be called {\it $(\alpha,\beta) $-set}.
\end{defin}

Let us list some immediate consequences of the introduced definitions.
If $\mathcal{A}$ is a $(\alpha,\beta)$-set, then $\alpha\leqslant\beta$. Any super-$\alpha$-set is a $(\alpha,d)$-set.
Each $(\alpha,\alpha)$-set $\mathcal{A}$ has combinatorial dimension $\alpha$; we will say that such a set has {\it exact} combinatorial dimension $\alpha$.
Note also that for any $1\le\alpha<\beta\le d$ there exists a $(\alpha,\beta)$-set that is not a $(\alpha',\beta')$-set whenever at least one of the inequalities $\alpha<\alpha'$ or $\beta>\beta'$ holds \cite[Theorem XIII.19]{Blei}.

\subsection{Systems of random unconditional convergence and divergence in a Banach space}
\label{rud}

Recall that a sequence $\{x_k\}_{k=1}^\infty$ of elements of a Banach space $X$ is called {\it basic} if it is a basis in its closed linear span. In the case when $\{x_{\pi(k)}\}_{k=1}^\infty$
is a basic sequence for any bijection $\pi:\,\mathbb{N}\to\mathbb{N}$,  $\{x_k\}_{k=1}^\infty$ is said to be a {\it unconditional basic sequence}. It is well known that a basic sequence $\{x_k\}_{k=1}^\infty$ in a Banach space $X$ is unconditional in $X$ if and only if there exists $D>0$ such that for any $n\in\mathbb{N}$, arbitrary collection of signs $\{\theta_k\}_{k=1} ^n$, $\theta_k=\pm1$, and all $a_k\in\mathbb{R}$ it holds
$$
{\left\|\sum_{k=1}^n\theta_ka_kx_k\right\|}_X\leqslant D{\left\|\sum_{k=1}^na_kx_k\right\|}_X.
$$
A detailed information about properties of basic and unconditional basic sequences can be found, for instance, in the books  \cite{AK,KashinSaakyan,Bra}.

Each of the next notions is a natural weakening of that of an unconditional basic sequence.

\begin{defin} 
A basic sequence $\{x_k\}_{k=1}^\infty$ in a Banach space $X$ is called a {\it system of random unconditional convergence with constant $D$} (in brief, $D$-RUC) (resp. {\it system of random unconditional divergence with constant $D$} (in brief, $D$-RUD)), where $D>0$, if for any $n\in\mathbb{N}$ and $a_k\in\mathbb{R}$, $k=1,2,\ldots,n$, we have
$$
\int\limits_0^1{\left\|\sum_{k=1}^nr_k(u)a_kx_k\right\|}_X\,du \leqslant D{\left\|\sum_{k=1}^na_kx_k\right\|}_X
$$
(resp.
$$
{\left\|\sum_{k=1}^na_kx_k\right\|}_X\leqslant D\int\limits_0^1{\left\|\sum_{k=1}^nr_k(u)a_kx_k\right\|}_X\,du).
$$
Moreover, {\it RUC} (resp. {\it RUD) system} is a $D$-RUC (resp. $D$-RUD) system with some constant $D$, the exact value of which is not important for us.
\end{defin}

The concept of a RUC system was introduced in the paper \cite{BKPS}, where also many important properties of such systems were proved. Subsequently, the behaviour of RUC and RUD systems in various function spaces was intensively studied by many authors (see e.g. \cite{wojtaszczyk86, garling-tomczak-jaegermann, dodds-semenov-sukochev, LAT, astashkin-curbera-tikhomirov, astashkin-curbera}).

It is clear that a basic sequence is unconditional in a Banach space if and only if it is both a RUC and a RUD sequence there (see also \cite[Proposition~2.3]{LAT}). Moreover, it follows easily from the definitions that a basic sequence is $1$-RUC (resp. $1$-RUD) if and only if it is $1$-unconditional \cite[Propositions~2.7 and 2.8]{LAT}.

Let $d\in\mathbb{N}$. By $\bigtriangleup^d$ we will denote the 'lower triangular' subset of the set $\mathbb{N}^d$, that is,
$$
\bigtriangleup^d:=\{(j_1,j_2,\ldots,j_d)\in\mathbb{N}^d:\;j_1>j_2>\ldots>j_d\}.
$$
Throughout the paper, by symbols $\jmath$ we denote multi-indices $(j_1,j_2,\ldots,j_d)\in \bigtriangleup^d$.  Furthermore, $\{ r_\jmath\}_{\jmath\in \bigtriangleup^d}$ is the Rademacher sequence, numbered in some (fixed) order by multi-indices $\jmath\in \bigtriangleup^d$, while $\mathbf{r}_\jmath(t):=r_{j_1}(t)\cdot r_{j_2}(t)\cdot\ldots\cdot r_{j_d}(t)$, $\jmath=(j_1,j_2,\ldots,j_d)\in \bigtriangleup^d$, where $r_j$, $j=1,2,\dots$, are the Rademacher functions, numbered as usual by positive integers (see \S\,\ref{Intro}).
Observe that the system $\{\mathbf{r}_\jmath\}_{\jmath\in \bigtriangleup^d}$, considered in the lexicographic order of $\jmath\in \bigtriangleup^d$, is basic in any symmetric space $X$ (see \cite[Theorem 2]{AL_AA}). However, in this paper, the numbering order of the system $\{\mathbf{r}_\jmath\}_{\jmath\in \bigtriangleup^d}$ will do not matter.

\section{Main results}

Our first result, which plays a key role in this paper, shows that, under certain non-restrictive conditions on the density characteristics of an index set, the RUD property of the corresponding subsystem of the Rademacher chaos in a symmetric space $X$ ensures  that $X$ is located sufficiently "far"\:from the space $L_\infty$.  

\begin{theorem}\label{RUDgen} Let $X$ be a symmetric space, $d\in\mathbb{N}$, $\alpha,\beta,b\in\mathbb{R}$, $1\le \alpha,\beta,b\le d$, $\alpha+b/\beta>b+1$. Assume also that 
$\mathcal{A}\subset\bigtriangleup^d$ is a $(\alpha,\beta)$-set such that for some $D>0$ and any finite set $\mathcal{A}'\subset \mathcal{A}$ the following inequality holds
\begin{equation}\label{eq4}
\left\|\sum_{\jmath\in \mathcal{A}'}\mathbf{r}_\jmath\right\|_X\leqslant D\int\limits_0^1 {\left\|\sum_{\jmath\in \mathcal{A}'}r_\jmath(u) \mathbf{r}_\jmath\right\|}_X\,du.
\end{equation}
Then, $X\supset \mathrm{Exp}L^{2/b}$.

In particular, the assertion is true if $\mathcal{A}$ is a $(\alpha-\varepsilon,\alpha+\varepsilon)$-set for some $\alpha>b$ and sufficiently small $\varepsilon>0$, and the system $\{\mathbf{r}_\jmath\}_{\jmath\in\mathcal{A}}$ is a RUD sequence in $X$.
\end{theorem}

\begin{proof} Observe first that the functions $\varphi(t)=\log^{-b/2}(e/t)$ and $M(u)=\exp(u^{2/b})-1$ satisfy conditions  \eqref{RavFundFunk} and \eqref{RavProstranstv}. Therefore,  
$\mathrm{Exp}L^{2/b}= \mathcal{M}(\log^{-b/2}(e/t))$. Since for any $\gamma>{b}/{2}$ the space $\mathcal{M}(\log^{-b/2}(e/t))$ is continuously embedded into the Lorentz space $\Lambda(\log^{-\gamma}(e/t))$ (see \cite[Corollary~1]{AL_AA}), the theorem will be proved once we show that $\Lambda(\log^{-\gamma}(e/t))\subset X$ for some $\gamma>{b}/{2}$.

It follows from the conditions of the theorem that $\alpha>1$. We choose $\alpha_0\in(1,\alpha)$ so that $\alpha_0+b/\beta>b+1$. Then, by the assumption, for each sufficiently large $n\in\mathbb{N}$ there are sets $B_1,B_2,\ldots, B_d$ such that $|B_j|=n$, $j=1,2,\ldots, d$, and
$$
|\mathcal{A}\cap \mathcal{B}_n|\geqslant n^{\alpha_0},
$$
where $\mathcal{B}_n:=B_1\times B_2\times\ldots \times B_d$. Fix
$n$ and a set $\mathcal{B}_n$ satisfying the above conditions. Since $|\mathcal{A}\cap \mathcal{B}_n|\leqslant n^d$, there exists $\delta\in[\alpha_0,d]$ that depends on $n$ and $\mathcal{ B}_n$ and also 
\begin{equation}
\label{card}
|\mathcal{A}\cap \mathcal{B}_n|=n^\delta.
\end{equation}

Then, we claim that there is a set $U_n\subset[0,1]$ such that $\mu(U_n)>1-2(e/2)^{-dn}$ and for all $u\in U_n$ 
\begin{equation}
\label{Linfty-Bound}
{\left\|\sum_{\jmath\in \mathcal{A}\cap \mathcal{B}_n}r_\jmath(u) \mathbf{r}_\jmath\right\|}_\infty\leqslant \sqrt{2d}n^\frac{\delta+1}{2}.
\end{equation}
Indeed, due to \eqref{card} and Bernstein's inequality (see e.g. \cite[Chapter~1, \S~6, formula (42)]{Shi} or \cite[Proposition~1.2]{AsBook}), we have for any $t\in[0,1]$ and $\lambda>0$
$$
\mu\left\{u\in[0,1]:\;\left\vert\sum_{\jmath\in \mathcal{A}\cap \mathcal{B}_n}{r}_\jmath(u) \mathbf{r}_{\jmath}(t)\right\vert>\lambda\right\}<2e^{-\frac{\lambda^2}{2n^\delta}},
$$
which implies
$$
\mu\left\{u\in[0,1]:\;\left\vert\sum_{\jmath\in \mathcal{A}\cap \mathcal{B}_n}{r}_\jmath(u) \mathbf{r}_{\jmath}(t)\right\vert>\sqrt{2d}n^{\frac{\delta+1}{2}}\right\}<2e^{-dn}.
$$
Observe that the sequence $\{\mathbf{r}_{\jmath}\}_{\jmath\in \mathcal{A}\cap \mathcal{B}_n}$ contains at most $dn$ distinct Rademacher functions. Therefore, there are at most $2^{dn}$ variants of the values of the sequence $\{\mathbf{r}_{\jmath}(t)\}_{\jmath\in \mathcal{A}\cap \mathcal{B}_n}$, where $t$ runs $[0,1]$. Therefore, from the preceding estimate it follows
$$
\mu\left\{u:\;\left\vert\sum_{\jmath\in \mathcal{A}\cap \mathcal{B}_n}{r}_\jmath(u) \mathbf{r}_{\jmath}(t)\right\vert>\sqrt{2d}n^{\frac{\delta+1}{2}}\;\mbox{for some }t\in [0,1]\right\}<2^{dn}\cdot2e^{-dn}.
$$
If now $U_n$ is the complement of the set from the last estimate, then $\mu(U_n)>1-2(e/2)^{-dn}$ and for all $u\in U_n$ we have \eqref{Linfty-Bound}. Thus, the claim is proved.

Next, thanks to \eqref{card}, for all $u\in[0,1]$ we have
$$
{\left\|\sum_{\jmath\in \mathcal{A}\cap \mathcal{B}_n}r_\jmath(u) \mathbf{r}_\jmath\right\|}_\infty\leqslant n^\delta.
$$
Consequently, by \eqref{Linfty-Bound}, it holds
\begin{eqnarray*}
\int\limits_0^1 {\left\|\sum_{\jmath\in \mathcal{A}\cap \mathcal{B}_n}r_\jmath(u) \mathbf{r}_\jmath\right\|}_\infty\,du &\leqslant & \int\limits_{[0,1]\setminus U_n} {\left\|\sum_{\jmath\in \mathcal{A}\cap \mathcal{B}_n}r_\jmath(u) \mathbf{r}_\jmath\right\|}_\infty\,du+\int\limits_{U_n} {\left\|\sum_{\jmath\in \mathcal{A}\cap \mathcal{B}_n}r_\jmath(u) \mathbf{r}_\jmath\right\|}_\infty\,du\\
&\leqslant & n^\delta\cdot 2(2/e)^{dn}+\sqrt{2d}n^\frac{\delta+1}{2},
\end{eqnarray*}
whence
\begin{equation}\label{ERavnNorm}
\int\limits_0^1 {\left\|\sum_{\jmath\in \mathcal{A}\cap \mathcal{B}_n}r_\jmath(u) \mathbf{r}_\jmath\right\|}_\infty\,du \leqslant C n^\frac{\delta+1}{2},
\end{equation}
where a constant $C$ depends only on $d$.

On the other hand, for some set of points $t\in[0,1]$ of measure $2^{-dn}$, all the Rademacher functions, products of which are the functions $\mathbf{r}_\jmath$, $\jmath\in \mathcal{A}\cap \mathcal{B}_n$, take the value $1$, and hence from \eqref{card} it follows
\begin{equation}\label{eq3}
{\left\|\sum_{\jmath\in \mathcal{A}\cap \mathcal{B}_n}\mathbf{r}_\jmath\right\|}_X\geqslant{\|n^\delta\chi_{(0,2^{-dn})}\|}_X\geqslant n^\delta\phi_X(2^{-dn}),
\end{equation}
where $\phi_X$ is the fundamental function of $X$. Taking into account now successively  condition \eqref{eq4}, the embedding $L_\infty\subset X$, estimate \eqref{ERavnNorm} and the inequality $\alpha_0\le\delta$, we obtain
\begin{eqnarray*}
\phi_X(2^{-dn})&\leqslant & n^{-\delta}{\left\|\sum_{\jmath\in \mathcal{A}\cap \mathcal{B}_n}\mathbf{r}_\jmath\right\|}_X\leqslant n^{-\delta}D\int\limits_0^1 {\left\|\sum_{\jmath\in \mathcal{A}\cap \mathcal{B}_n}r_\jmath(u) \mathbf{r}_\jmath\right\|}_X\,du\\
&\leqslant & n^{-\delta}C_1\int\limits_0^1 {\left\|\sum_{\jmath\in \mathcal{A}\cap \mathcal{B}_n}r_\jmath(u) \mathbf{r}_\jmath\right\|}_\infty\,du \leqslant C_2n^{-(\delta-1)/2}\leqslant C_2n^{-(\alpha_0-1)/2}.
\end{eqnarray*}
Since this inequality holds for all sufficiently large $n\in\mathbb{N}$ and the function $\phi_X$ is quasiconcave, then for all $t\in[0,1]$ we have
$$
\phi_X(t)\leqslant C\log^{-\gamma_0} (e/t),
$$
where $\gamma_0=(\alpha_0-1)/2>0$. Hence, $\Lambda(\log^{-\gamma_0}(e/t))\subset\Lambda(\phi_X)\subset X$, and if $\gamma_0>b/2$, i.e., if $\alpha_0>b+1$, the proof is complete. In the case when $\gamma_0\le b/2$ (equivalently, $\alpha_0\le b+1$), we proceed as follows.

According to Blei's inequalities (see \cite[Formula VII.(9.30) and Corollary XIII.29]{Blei} or \cite[Formula (1.7)]{BleiJanson}), for the same $\delta$ as above, all $p\geqslant 1$ and $u\in [0,1]$ it holds
$$
{\left\|\sum_{\jmath\in \mathcal{A}\cap \mathcal{B}_n}r_\jmath(u)\mathbf{r}_\jmath\right\|}_p\leqslant Cp^{\frac{\beta}{2}}\left(\sum_{\jmath\in \mathcal{A}\cap \mathcal{B}_n}(r_\jmath(u))^2\right)^\frac{1}{2}=Cp^{\frac{\beta}{2}}n^{\frac{\delta}{2}}.
$$
Therefore, by using extrapolation description \eqref{description} of the exponential Orlicz space $\mathrm{Exp}L^{2/\beta}$, we conclude that
$$
{\left\|\sum_{\jmath\in \mathcal{A}\cap \mathcal{B}_n}r_\jmath(u) \mathbf{r}_\jmath\right\|}_{\mathrm{Exp}L^{2/\beta}}\leqslant Cn^{\frac{\delta}{2}}.
$$
Hence, the equality $\mathrm{Exp}L^{2/\beta}=\mathcal{M}(\log^{-\beta/2}(e/t))$ and the definition of the norm in Marcinkiewicz spaces (see \S\,\ref{Symm}) imply that for all $u\in [0,1]$ 
$$
\left(\sum_{\jmath\in \mathcal{A}\cap \mathcal{B}_n}r_\jmath(u) \mathbf{r}_\jmath\right)^*(t)\leqslant Cn^{\delta/2}\log^{\beta/2}(e/t),\;\;0<t\le 1.
$$
Combining the latter inequality together with \eqref{Linfty-Bound}, we get for all $u\in U_n$ the estimate
\begin{equation}\label{eq2}
\left(\sum_{\jmath\in \mathcal{A}\cap \mathcal{B}_n}r_\jmath(u) \mathbf{r}_\jmath\right)^*(t)\leqslant Cn^{\delta/2}\min\{n^{1/2},\log^{\beta/2}(e/t)\},\;\;0<t\le 1.
\end{equation}

Further, setting $\gamma_{n+1}=\gamma_0+\gamma_n/\beta$, $n=0,1,\dots$, where still $\gamma_0=(\alpha_0-1)/2$, we will show that for each $n=0,1,\dots$
\begin{equation}\label{embeddings}
\Lambda(\log^{-\gamma_n}(e/t))\subset X.
\end{equation}
Since this embedding is valid for $n=0$, it suffices to check that  the validity of \eqref{embeddings} for $n=k$ implies that for $n=k+1$. 

Indeed, 
from inequalities \eqref{eq3}, \eqref{eq4} and \eqref{eq2} it follows
\begin{eqnarray*}
\phi_X(2^{-dn})&\leqslant & n^{-\delta}{\left\|\sum_{\jmath\in \mathcal{A}\cap \mathcal{B}_n}\mathbf{r}_\jmath\right\|}_X\leqslant Dn^{-\delta}\int\limits_0^1 {\left\|\sum_{\jmath\in \mathcal{A}\cap \mathcal{B}_n}r_\jmath(u) \mathbf{r}_\jmath\right\|}_X\,du\\
&\leqslant & Cn^{-\delta}\left(\int\limits_{[0,1]\setminus U_n} {\left\|\sum_{\jmath\in \mathcal{A}\cap \mathcal{B}_n}r_\jmath(u) \mathbf{r}_\jmath\right\|}_\infty du+\int\limits_U {\left\|\sum_{\jmath\in \mathcal{A}\cap \mathcal{B}_n}r_\jmath(u) \mathbf{r}_\jmath\right\|}_{\Lambda(\log^{-\gamma_k}(e/t))} du\right)\\
&\leqslant & C\cdot 2(e/d)^{-dn}\\
&\ &\quad+C'n^{-\delta/2}\left(\int\limits_0^{e^{1-n^{1/\beta}}}n^{\frac{1}{2}}\,d\log^{-\gamma_k}(e/t)+\int\limits_{e^{1-n^{1/\beta}}}^1\log^{\frac{\beta}{2}}(e/t)\,d\log^{-\gamma_k}(e/t)\right)\\
&\leqslant & C''n^{-(\delta/2+\gamma_k/\beta-1/2)}\leqslant C''n^{-(\alpha_0/2+\gamma_k/\beta-1/2)}=C''n^{-\gamma_{k+1}}.
\end{eqnarray*}
Hence, taking into account the quasiconcavity of the fundamental functions, we immediately obtain \eqref{embeddings} for $n=k+1$.

Next, note that
$$
\gamma_n=\gamma_0\sum_{i=0}^{n}\frac{1}{\beta^i}\to \frac{\beta\gamma_0}{\beta-1}\;\;\mbox{as}\;\;n\to\infty.
$$
Moreover, thanks to the assumptions, we have $\alpha_0>b+1-b/\beta$ and $\beta\ge 1$. In consequence,
$$
\frac{\beta\gamma_0}{\beta-1}>\frac12\Big(b-\frac{b}{\beta}\Big)\frac{\beta}{\beta-1}=\frac{b}{2}.
$$
From the last relations it follows that the inequality $\gamma_n>b/2$   holds whenever $n$ is sufficiently large. As was observed in the very beginning of the proof, this inequality combined with embedding \eqref{embeddings} implies the desired result. 
\end{proof}

In the case when $b=1$ we get

\begin{cor}\label{RUD} Let $X$ be a symmetric space and let $d\in\mathbb{N}$. Suppose that 
$\mathcal{A}\subset\bigtriangleup^d$ is a $(\alpha,\beta)$-set, with $\alpha+1/\beta>2$, such that for some $D>0$ and any finite set $\mathcal{A}'\subset \mathcal{A}$ it holds
\begin{equation*}
\left\|\sum_{\jmath\in \mathcal{A}'}\mathbf{r}_\jmath\right\|_X\leqslant D\int\limits_0^1 {\left\|\sum_{\jmath\in \mathcal{A}'}r_\jmath(u) \mathbf{r}_\jmath\right\|}_X\,du.
\end{equation*}
Then, $\mathrm{Exp}L^{2}\subset X$.

In particular, the latter embedding is fulfilled if $\mathcal{A}$ is a $(\alpha-\varepsilon,\alpha+\varepsilon)$-set for some $\alpha>1$ and sufficiently small $\varepsilon>0$ and the system $\{\mathbf{r}_\jmath\}_{\jmath\in\mathcal{A}}$ is a RUD sequence in $X$.
\end{cor}

\begin{theorem}\label{FractChaos} Let $X$ be a symmetric space and let $d\in\mathbb{N}$. Assume that 
$\mathcal{A}\subset\bigtriangleup^d$ is a $(\alpha,\beta)$-set, $\alpha+1/\beta>2$. The following conditions are equivalent:\\
(a) $\{\mathbf{r}_\jmath\}_{\jmath\in \mathcal{A}}$ is a RUD sequence in $X$;\\
(b) $\{\mathbf{r}_\jmath\}_{\jmath\in \mathcal{A}}$ is an unconditional basic sequence in $X$;\\
(c) the sequence $\{\mathbf{r}_\jmath\}_{\jmath\in \mathcal{A}}$ is equivalent in $X$ to the canonical basis in $\ell_2$, that is, for some constant $C_X$ we have
\begin{equation}
\label{equiv}
C_X^{-1}{\left\|\{a_{\jmath}\}_{\jmath\in\mathcal{A}}\right\|}_{\ell_2}\leqslant{\left\|\sum_{\jmath\in\mathcal{A}}a_{\jmath}\mathbf{r}_\jmath\right\|}_X\leqslant C_X{\left\|\{a_{\jmath}\}_{\jmath\in\mathcal{A}}\right\|}_{\ell_2}.
\end{equation}

In particular, if $\alpha>1$, then for sufficiently small $\varepsilon>0$ and arbitrary $(\alpha-\varepsilon,\alpha+\varepsilon)$-set $\mathcal{A}$ conditions (a), (b) and (c) are equivalent.
\end{theorem}

It is clear that only the implication $(a)\Rightarrow (c)$ needs a proof. This proof follows directly from Corollary \ref{RUD} and the following assertion.

\begin{prop}\label{OrlSq} Suppose $X$ is a symmetric space such that $\mathrm{Exp}L^2\subset X$. Then there is a constant $C'$ such that for each uniformly bounded $D$-RUD sequence $\{x_j\}_{j\in\mathbb{N}}$ from $X$ the following Khintchine type inequality holds:
$$
{\left\|\sum_{j\in\mathbb{N}} a_jx_j\right\|}_X\leqslant C'D\sup_{j\in\mathbb{N}}{\|x_j\|}_{\infty}\cdot\left(\sum_{j\in\mathbb{N}}a_j^2\right)^\frac{1}{2}.
$$
\end{prop}
\begin{proof}
As is known (see e.g. \cite[Lemma 3]{AL_AA}), for every Orlicz function $M$ and any measurable function $z=z(u,t)$ defined on the square $[0,1]\times[0,1]$ we have
\begin{equation}\label{eq0}
\int\limits_0^1{\|z(u,\cdot)\|}_{\mathrm{L}_M(\cdot)}\,du\leqslant 2\;\mathrm{ess}\sup_{t\in[0,1]}{\|z(\cdot,t)\|}_{\mathrm{L}_M(\cdot)}.
\end{equation}
Therefore, from the conditions of the proposition and the Khintchine inequality applied to the Rademacher system in the space $\mathrm{Exp}L^2$ (see \cite[Chapter~V, Theorem~8.7]{Zyg} or \cite{RS}) it follows
\begin{eqnarray*}
{\left\|\sum_{j\in\mathbb{N}}a_jx_j\right\|}_X &\leqslant & D\int\limits_0^1{\left\|\sum_{j\in\mathbb{N}}r_j(u)a_jx_j\right\|}_X\,du \\ 
&\leqslant & DC\int\limits_0^1{\left\|\sum_{j\in\mathbb{N}}r_j(u)a_jx_j(\cdot)\right\|}_{\mathrm{Exp}L^2(\cdot)}\,du\\
&\leqslant & 2 DC\;\mathrm{ess}\sup_{t\in[0,1]}{\left\|\sum_{j\in\mathbb{N}}r_j(\cdot)a_jx_j(t)\right\|}_{\mathrm{Exp}L^2(\cdot)}\\
&\leqslant & C'D\;\mathrm{ess}\sup_{t\in[0,1]} \left(\sum_{j\in\mathbb{N}}(a_jx_j(t))^2\right)^\frac{1}{2}\\
&\leqslant & C'D\sup_{j\in\mathbb{N}}{\|x_j\|}_{\infty}\cdot\left(\sum_{j\in\mathbb{N}}a_j^2\right)^\frac{1}{2}.
\end{eqnarray*}

\end{proof}

Theorem \ref{FractChaos} shows that the difference in the behaviour of the Rademacher sequence $\{r_j\}$ and the chaos $\{r_{j_1}r_{j_2}\}_{j_1>j_2}$, noted in the Introduction, is caused by different combinatorial dimensions of these systems. Moreover, this result implies that unconditionality of a subsystem $\{\mathbf{r}_\jmath\}_{\jmath\in \mathcal{A}}$ of the chaos of any order $d$ in a  symmetric space $X$ and its equivalence in $X$ to the canonical basis in $\ell_2$ are equivalent whenever the corresponding index set $\mathcal{A}$ has exact combinatorial dimension $\alpha>1$.

In the case of Orlicz spaces Theorem \ref{FractChaos} may be refined. Namely, if a set $\mathcal{A}$ has exact combinatorial dimension $\alpha>1$, the above conditions (a), (b) and (c) can be characterized in terms of certain embeddings.

\begin{theorem}\label{FractChaosOrl} Let $L_M$ be an Orlicz space, $d\in\mathbb{N}$. Suppose that a set $\mathcal{A}\subset\bigtriangleup^d$ has exact combinatorial dimension $\alpha>1$. The following conditions are equivalent: 
\\
(i) $\{\mathbf{r}_\jmath\}_{\jmath\in \mathcal{A}}$ is a RUD sequence in $L_M$;\\
(ii) $\{\mathbf{r}_\jmath\}_{\jmath\in \mathcal{A}}$ is an unconditional basic sequence in $L_M$;\\
(iii) the sequence $\{\mathbf{r}_\jmath\}_{\jmath\in \mathcal{A}}$ is equivalent in $L_M$ to the canonical basis in $\ell_2$, that is, for some constant $C_M$ we have
$$
C_M^{-1}{\left\|\{a_{\jmath}\}_{\jmath\in \mathcal{A}}\right\|}_{\ell_2}\leqslant{\left\|\sum_{\jmath\in \mathcal{A}}a_{\jmath}\mathbf{r}_\jmath\right\|}_{L_M}\leqslant C_M{\left\|\{a_{\jmath}\}_{\jmath\in \mathcal{A}}\right\|}_{\ell_2};
$$
(iv) $L_M\supset\mathrm{Exp}L^{2/\alpha}$.
\end{theorem}
\begin{proof} 
The equivalence of conditions (i), (ii) and (iii) follows from Theorem \ref{FractChaos}. Thus, it remains only to verify that (iii) is equivalent to (iv).

Assume first that embedding (iv) holds. Applying again Blei's inequalities (\cite[Formula VII.(9.30) and Corollary XIII.29]{Blei} or \cite[Formula (1.7)]{BleiJanson}), for all $p\geqslant 1$ and any  sequence $\{a_\jmath\}_{\jmath\in A}$ we obtain
$$
{\left\|\sum_{\jmath\in \mathcal{A}}a_\jmath\mathbf{r}_\jmath\right\|}_p\leqslant C(\alpha,d)p^{\frac{\alpha}{2}}\left(\sum_{\jmath\in \mathcal{A}}a_\jmath^2\right)^\frac{1}{2}.
$$
Therefore, by the embedding $L_M\supset\mathrm{Exp}L^{2/\alpha}$   
and the extrapolation description of the space ${\rm Exp} L^{2/\alpha}$ (see \eqref {description}), we have 
$$
{\left\|\sum_{\jmath\in \mathcal{A}}a_\jmath \mathbf{r}_\jmath\right\|}_{L_M}\leqslant C{\left\|\sum_{\jmath\in \mathcal{A}}a_\jmath \mathbf{r}_\jmath\right\|}_{\mathrm{Exp}L^{2/\alpha}}\leqslant C'\left(\sum_{\jmath\in \mathcal{A}}a_\jmath^2\right)^\frac{1}{2},
$$
and hence the right-hand side inequality in (iii) follows. Since $X\subset L_1$, the left-hand side of this inequality is fulfilled in every symmetric space $X$ (see also \cite[Lemma 6]{AL_AA}). Thus, the implication (iv)$\Rightarrow$ (iii) is proven.


Conversely, let us prove the implication (iii)$\Rightarrow$ (iv).  Since by assumption the set $\mathcal{A}$ has exact combinatorial dimension $\alpha$, then for some constant $C>0$ and each $n\in\mathbb{N}$ there is a set $\mathcal{B} _n:=B_1\times B_2\times\ldots \times B_d$ such that $|B_j|=n$, $j=1,2,\ldots,d$, and
$$
C^{-1}n^{\alpha}\leqslant |\mathcal{A}\cap \mathcal{B}_n|\leqslant Cn^{\alpha}.
$$
Thus, in view of \eqref{eq3} (with $\alpha$ instead of $\delta$) and condition  (iii), we have 
$$
\phi_{L_M}(2^{-dn})\leqslant Cn^{-\alpha}{\left\|\sum_{\jmath\in \mathcal{A}\cap\mathcal{B}_n}\mathbf{r}_\jmath\right\|}_{L_M}\leqslant Cn^{-\alpha}C_M\Big(\sum_{\jmath\in \mathcal{A}\cap\mathcal{B}_n} 1\Big)^{1/2}\leqslant C'n^{-\alpha/2}.
$$
Consequently, taking into account the quasiconcavity of $\phi_{L_M}$, we conclude that
$$
\phi_{L_M}(t)\leqslant C\log^{-\alpha/2}(e/t),\;\;t\in(0,1],
$$
with some constant $C$. Since $\phi_{L_M}(t)=1/M^{-1}(1/t)$ (see \S\,\ref{Symm}), the latter inequality implies
$$
\log^{\alpha/2}(e/t)\leqslant C M^{-1}(1/t),$$
or equivalently
$$
M(c\log^{\alpha/2}(e/t))\leqslant 1/t.
$$
As a result, after the change $c\log^{\alpha/2}(e/t)=u$, we arrive at the inequality
$$
M(u)\leqslant e^{(Cu)^{2/\alpha}-1}\quad\mbox{ for }u\geqslant 1.
$$
By the definition of the norm in Orlicz spaces (see \S\,\ref{Symm}), this implies immediately that $X\supset\mathrm{Exp}L^{2/\alpha}$, and the proof of (iv) is completed. 
\end{proof}

\section{Concluding remarks}

\subsection{On the RUC property of uniformly bounded Bessel systems in symmetric spaces}

According to Theorem \ref{RUDgen}, under certain conditions on density characteristics of an index set, the assumption that the corresponding subsystem of the Rademacher chaos has the RUD property in a symmetric space $X$ implies that $X$ is located "far"\:from the space $L_\infty$.
In a certain sense, the opposite assertion is valid for the random unconditional convergence (RUC) property (see \S\,\ref{rud}) of such a subsystem. We obtain this result as a consequence of a more general statement related to uniformly bounded Bessel systems of functions. A similar assertion was proved, under extra conditions that $X\subset L_2$ and a system is orthonormal, in the paper \cite[Proposition~2.1]{dodds-semenov-sukochev} (see also \cite[Corollary~1.4]{BKPS}). 

Recall that a bounded basic sequence $\{x_j\}_{j\in\mathbb{N}}$ in a Banach space $X$ is called a {\it Bessel} system if for some constant $C(X)$ and any $ a_j\in\mathbb{R}$, $j\in\mathbb{N}$, we have
$$
\left(\sum_{j\in\mathbb{N}}a_j^2\right)^\frac{1}{2}\leqslant
C(X)\left\|\sum_{j\in\mathbb{N}} a_jx_j\right\|_X.
$$

\begin{prop}\label{RUC} Suppose $X$ is a symmetric space such that $\mathrm{Exp}L^2\subset X$. Then every uniformly bounded Bessel sequence $\{x_j\}_{j\in\mathbb{N}}$ has the RUC property in $X$.
\end{prop}
\begin{proof}

%
By the conditions of the proposition, inequality \eqref{eq0} for $L_M=\mathrm{Exp}L^2$  and the Khintchine inequality in the space $\mathrm{Exp}L^2$ (see \cite[ Chapter~V, Theorem~8.7]{Zyg} or \cite{RS}) we get
\begin{eqnarray*}
\int\limits_0^1{\left\|\sum_{j\in\mathbb{N}}r_j(u)a_jx_j\right\|}_X\,du  
&\leqslant & C'\int\limits_0^1{\left\|\sum_{j\in\mathbb{N}}r_j(u)a_jx_j(\cdot)\right\|}_{\mathrm{Exp}L^2(\cdot)}\,du\\
&\leqslant & 2 C'\;\mathrm{ess}\sup_{t\in[0,1]}{\left\|\sum_{j\in\mathbb{N}}r_j(\cdot)a_jx_j(t)\right\|}_{\mathrm{Exp}L^2(\cdot)}\\
&\leqslant & C''\;\mathrm{ess}\sup_{t\in[0,1]} \left(\sum_{j\in\mathbb{N}}(a_jx_j(t))^2\right)^\frac{1}{2}\\
&\leqslant & C''\sup_{j\in\mathbb{N}}{\|x_j\|}_{\infty}\cdot\left(\sum_{j\in\mathbb{N}}a_j^2\right)^\frac{1}{2}\\
&\leqslant & C''C(X)\sup_{j\in\mathbb{N}}{\|x_j\|}_{\infty}\cdot \left\|\sum_{j\in\mathbb{N}} a_jx_j\right\|_X.
\end{eqnarray*}

\end{proof}

Since 
$\{\mathbf{r}_\jmath\}_{\jmath\in\bigtriangleup^d}$ is an uniformly bounded orthonormal sequence on $[0,1]$, the proposition \ref{RUC} implies the following consequence.

\begin{cor}
$\{\mathbf{r}_\jmath\}_{\jmath\in\bigtriangleup^d}$ is a RUC sequence in every symmetric space $X$ such that $\mathrm{Exp}L^2\subset X$.
\end{cor}

\subsection{Asymptotic independence of a fractional Rademacher chaos}

Let $d=3$, $\mathcal{A}=\{(i,j,i+j),1\leqslant i<j\}$. Then, as one can easily see, $\mathcal{A}$ is a $(2,2)$-set, and therefore, by Theorem \ref{FractChaosOrl}, we have 
$$
{\left\|\sum_{\jmath\in \mathcal{A}}a_{\jmath}\mathbf{r}_\jmath\right\|}_{\mathrm{Exp}L}\asymp{\left\|\{a_{\jmath}\}_{\jmath\in \mathcal{A}}\right\|}_{\ell_2}
$$
and
$$
\sup\left\{\left\|\sum_{\jmath\in E}a_{\jmath}\mathbf{r}_\jmath\right\|_{\mathrm{Exp}L^\gamma}:\,\left\|\{a_{\jmath}\}_{\jmath\in E}\right\|_{\ell_2}\leqslant 1,\,E\subset \mathcal{A}\;\;\mbox{is finite}\;\;\right\}=\infty$$
for every $\gamma>1$. 
Moreover, if $\mathcal{A}_N:=\mathcal{A}\cap \{1,2,\dots,N\}^3$, where $N\in\mathbb{N}$, $N\geqslant 3$, then the sums
$$
S_N:=|\mathcal{A}_N|^{-1/2}\sum_{\jmath\in\mathcal{A}_N}\mathbf{r}_\jmath$$
are normed in $L_2$ and, by \cite[Theorem~1.5]{BleiJanson}, satisfy the equivalence 
$$
\sup_{N}\|S_N\|_p\asymp p,\;\;p\ge 1.$$
Consequently, \eqref{description} implies
$$
\inf\biggl\{\gamma:\,\sup_{N}{\|S_N\|}_{\mathrm{Exp}L^\gamma}=\infty\biggr\}=1.
$$

The latter relation can be considered as a consequence of a certain "interdependence"\:of the functions $\mathbf{r}_\jmath$, $\jmath\in \mathcal{A}$. In contrast to that, as we will see below, asymptotically the sums $S_N$ have the standard normal distribution, 
and thus the functions $\mathbf{r}_\jmath$, $\jmath\in \mathcal{A}$, are asymptotically independent like the usual Rademacher functions. This "divergence"\;in estimates for the  moments of a Rademacher fractional chaos and its asymptotic behaviour was previously observed in the paper \cite{BleiJanson}. To justify the last assertion, we will use Theorem 1.7 from \cite{BleiJanson}.

Let 
$$
\mathcal{A}_{N,k}^*:=\{(i,j,m)\in \mathcal{A}_N:\,k\in \{i,j,m\}\},\;\;k\in\mathbb{N}.$$
Consider also the set $\mathcal{A}_N^\sharp\subset \mathcal{A}_N\times \mathcal{A}_N$ consisting of pairs $((i,j,i+j),(k,l ,k+l))$ of elements of the set $\mathcal{A}_N$ such that
\begin{equation}\label{nothing}
\{i,j,i+j\}\cap\{k,l,k+l\}=\varnothing
\end{equation}
 and
 \begin{equation}\label{nothing1}
\{i,j,i+j,k,l,k+l\}=\{i_1,j_1,i_1+j_1,k_1,l_1,k_1+l_1\}
\end{equation}
for some $(i_1,j_1,i_1+j_1),(k_1,l_1,k_1+l_1)\in \mathcal{A}_N$, satisfying the conditions
\begin{equation}\label{intersection}
(i_1,j_1,i_1+j_1)\ne (i,j,i+j)\;\;\mbox{and}\;\; (i_1,j_1,i_1+j_1)\ne (k,l,k+l).
\end{equation}

To prove that the sums $S_N$ have asymptotically the standard normal distribution, it suffices to check the following relations: 
$$
\lim_{N\to\infty}\max_{k} \frac{|\mathcal{A}_{N,k}^*|}{|\mathcal{A}_N|}=0$$
and
$$
\lim_{N\to\infty}\frac{|\mathcal{A}_N^\sharp|}{|\mathcal{A}_N|^2}=0$$
\cite[Theorem 1.7]{BleiJanson}. Observe that the first of these equalities is a consequence of the obvious estimates: $|\mathcal{A}_{N,k}^*|\leqslant 3N$ and $|\mathcal{A}_N|\asymp N^2$. To verify the second, it suffices to show that $\mathcal{A}_N^\sharp=\varnothing$.





Assume that $((i,j,i+j),(k,l,k+l))\in\mathcal{A}_N^\sharp$, i.e., we have \eqref{nothing} as well as \eqref{nothing1} for some elements $(i_1,j_1,i_1+j_1),(k_1,l_1,k_1+l_1)\in \mathcal{A}_N$, satisfying the conditions \eqref{intersection}.
Let 
$$
S:=\{i,j,i+j,k,l,k+l\}=\{i_1,j_1,i_1+j_1,k_1,l_1,k_1+l_1\}.$$
Then, 
$$
\max\{x:\,x\in S\}=\max\{i+j,k+l\}=\max\{i_1+j_1,k_1+l_1\}
$$
and
$$
\Sigma_{S}=2(i+j+k+l)=2(i_1+j_1+k_1+l_1),
$$
where $\Sigma_{S}$ is the sum of all elements of the set  
$S$.
Therefore, either $i+j=i_1+j_1, k+l=k_1+l_1$, or $i+j=k_1+l_1, k+l=i_1+j_1$, whence
$$
\{i,j,k,l\}=\{i_1,j_1,k_1,l_1\}.
$$
By the assumption, the numbers $i,j,k,l,i+j,k+l$ are pairwise distinct (see \eqref{nothing}), so we have
$$ 
i+k\ne i+j,\;\; i+l\ne i+j,\;\; j+k\ne i+j,\;\; j+l\ne i+j,\;\; k+l\ne i+j.
$$
But then the equality $i_1+j_1=i+j$ implies $i_1=i, j_1=j$, which contradicts \eqref{intersection}. Similarly, from the equality $i_1+j_1=k+l$ it follows $i_1=k, j_1=l$, which also contradicts \eqref{intersection}. 

\end{document}